\documentclass[a4paper,reqno]{amsart}

\usepackage{amssymb}
\usepackage{amsmath}
\usepackage{amsthm}
\usepackage{enumitem}
\usepackage{xcolor}
\usepackage{graphicx}
\usepackage{pgfarrows,pgfnodes}
\usepackage{url}
\usepackage{textcomp}
\usepackage{tikz}
\usepackage{float}
\usepackage{hyperref}

\usepackage{fullpage}

\usepackage{texdraw,amstext,amsfonts,color,tabu,dsfont}

\usepackage{mathtools,tabularx}

\usepackage{float}

\definecolor{foge}{rgb}{0.1, 0.6, 0.1}

\newcommand{\Ll}{\Lambda}
\newcommand{\Z}{\mathbb{Z}}
\newcommand{\B}{\mathcal{B}}

\newcommand{\C}{\mathcal{C}}

\newcommand{\E}{\mathcal{E}}

\newcommand{\Co}{\mathcal{C}}

\newcommand{\Pp}{\mathcal{P}}

\newcommand{\Ppm}{\Pp^{\gg}_{\cmu}}
\newcommand{\Pppm}{\Pp^{\gtrdot}_{\cmu}}
\newcommand{\Psucm}{\Pp^{\succ}_{\cmu}}

\newcommand{\la}{\lambda}

\newcommand{\gf}{\mathfrak g}
\newcommand{\h}{\mathfrak h}
\newcommand{\ot}{\otimes}
\newcommand{\p}{\mathfrak p}

\newcommand{\cmu}{c_{g_0}\cdots c_{g_{t-1}}}

\DeclareMathOperator{\ch}{ch}
\DeclareMathOperator{\sgn}{sgn}
\DeclareMathOperator{\mult}{mult}
\DeclareMathOperator{\wt}{\overline{wt}}
\DeclareMathOperator{\wta}{wt}

\numberwithin{equation}{section}

\newtheorem{theo}{Theorem}[section]
\newtheorem{prop}[theo]{Proposition}
\newtheorem{rem}[theo]{Remark}
\theoremstyle{definition}
\newtheorem{deff}[theo]{Definition}

\title{Partition identities from higher level crystals of $A_1^{(1)}$}
\author{Jehanne Dousse}
\address{Université de Genève, Section de Mathématiques, 7-9 rue du Conseil-Général, CH-1205 Genève, Switzerland}
\email{jehanne.dousse@unige.ch}

\author{Leonard Hardiman}
\address{Chair of Statistical Field Theory, Institute of Mathematics, EPFL, Station 8, CH-1015 Lausanne, Switzerland}
\email{leonard.hardiman@epfl.ch}

\author{Isaac Konan}
\address{Université Claude Bernard Lyon 1, UMR5208, Institut Camille Jordan, F-69622 Villeurbanne, France}
\email{konan@math.univ-lyon1.fr}

\keywords{integer partitions, Rogers--Ramanujan type identities, Kac--Moody Lie algebras, character formulas, crystal bases}
\subjclass[2010]{05A15, 05A17, 05A30, 05E10, 11P81, 11P84, 17B10, 17B65, 17B67}

\begin{document}

\begin{abstract}
  We study perfect crystals for the standard modules of the affine Lie algebra $A_1^{(1)}$ at all levels using the theory of multi-grounded partitions. We prove a family of partition identities which are reminiscent of the Andrews--Gordon identities and companions to the Meurman--Primc identities, but with simple difference conditions involving absolute values.
\end{abstract}

\maketitle

\section{Introduction and statement of results}

The content of this paper lies in the intersection of combinatorics (in particular, the study of partition identities) and the representation theory of affine Kac--Moody algebras. Lepowsky, Milne and Wilson's foundational work~\cite{Le-Mi1,Lepowsky,Lepowsky2} has led to a productive relationship between these two fields. We now recall some background and provide a brief description of the principal mechanism through which this interaction takes place.

A partition of a natural number $n$ is a non-increasing sequence of positive integers whose sum is $n$. For example, the partitions of $4$ are $(4), (3,1), (2,2), (2,1,1)$ and $(1,1,1,1)$.
Among the most famous and ubiquitous partition and $q$-series identities are those of Rogers--Ramanujan \cite{RogersRamanujan}. In $q$-series form, they can be stated as follows:
\begin{theo}[The Rogers--Ramanujan identities]
  \label{th:RR}
  Let $i=0$ or $1$. Then
  $$
  \sum_{n \geq 0} \frac{q^{n^2+ (1-i)n}}{(q;q)_n} = \frac{1}{(q^{2-i};q^5)_{\infty}(q^{3+i};q^5)_{\infty}}.
  $$
\end{theo}
Here and throughout the paper, we use the standard $q$-series notation: for $n \in \mathbb{N} \cup \{\infty\}$ and $j \in \mathbb{N}$,
\begin{align*}
  (a;q)_n &:= \prod_{k=0}^{n-1} (1-aq^k),\\
  (a_1, \dots, a_j ; q)_n &:= (a_1;q)_n \cdots (a_j;q)_n.
\end{align*}

Interpreting the product side of the Rogers--Ramanujan identities as the generating function for partitions with congruence conditions and the sum side as the generating function for partitions with difference conditions yields the combinatorial identities:
\begin{theo}[Rogers--Ramanujan identities, combinatorial version]
  \label{th:RRcomb}
  Let $i=0$ or $1$. For every natural number $n$, the number of partitions of $n$ such that the difference between two consecutive parts is at least $2$ and the part $1$ appears at most $i$ times is equal to the number of partitions of $n$ into parts congruent to $\pm (2-i) \mod 5.$
\end{theo}

These identities were given many proofs and generalisations, see e.g. \cite{AndrewsGordon,Andrews89,Bressoud83,CorteelRSK,Garsiamilne,Gordon61,GOW16}, including the famous Andrews--Gordon identities ~\cite{AndrewsGordon}:
\begin{theo}[Andrews--Gordon identities]
  \label{th:AGseries}
  Let $r$ and $i$ be integers such that $r\geq 2$ and $1\leq i \leq r.$ Let $G_{i,r}(n)$ be the number of partitions $\lambda=(\lambda_1,\lambda_2,\dots,\lambda_s)$ of $n$ such that $\lambda_{j}-\lambda_{j+r-1} \geq 2$ for all $j$, and at most $i-1$ of the $\lambda_j$ are equal to $1$. Then
  $$
  \sum_{n\geq0} G_{i,r}(n) q^n=\frac{(q^{2r+1},q^{i},q^{2r-i+1};q^{2r+1})_\infty}{(q;q)_\infty}.
  $$
\end{theo}
Note that the Rogers--Ramanujan identities correspond to the particular case $r=2$ in Theorem \ref{th:AGseries}.

In~\cite{BressoudAG}, Bressoud found a counterpart for even moduli to the Andrews--Gordon identities.
\begin{theo}[Bressoud]
  \label{th:Bressoud}
  Let $r$ and $i$ be integers such that $r\geq 1$ and $1\leq i \leq r.$ Let $B_{i,r}(n)$ be the number of partitions $\lambda=(\lambda_1,\lambda_2,\dots,\lambda_s)$ of $n$ such that $\lambda_{j}-\lambda_{j+r-1} \geq 2$ for all $j$, if $\lambda_j \leq \lambda_{j+r-2} +1$ then $\lambda_j + \cdots + \lambda_{j+r-2} \equiv i-1 \mod 2$, and at most $i-1$ of the $\lambda_j$ are equal to $1$. We have
  $$
  \sum_{n\geq0} B_{i,r}(n) q^n=\frac{(q^{2r},q^{i},q^{2r-i};q^{2r})_\infty}{(q;q)_\infty}.
  $$
\end{theo}

We now explain how these partition identities are related to representation theory.
We assume that the reader is familiar with the notion of a Kac--Moody algebra, the standard reference text is~\cite{Kac}. Let $\gf$ be an affine Kac--Moody algebra with positive roots $\Delta^{+} \subset \h^{*}$, where $\h$ is a Cartan subalgebra of $\gf$. Let $P$ (resp.\ $P^{+}$) be the corresponding set of integral (resp.\ dominant integral) weights and let $L(\lambda)$ be the irreducible highest weight $\gf$-module of highest weight $\lambda \in P^+$ (also called the standard module with highest weight $\lambda$).
The character of $L(\lambda)$ is defined as
$$
\ch L(\lambda) = \sum_{\mu\in P} \dim L(\lambda)_{\mu} \cdot e^{\mu},
$$
where $e$ is a formal exponential and $L(\lambda)_{\mu}$ is the weight space corresponding to $\mu$. A well studied problem within representation theory is the development of so-called character formulas i.e.\ explicit algebraic expressions for $\ch L(\lambda)$. A well-known example is provided by the Weyl--Kac character formula \cite[Proposition 10.10]{Kac}
\begin{align} \label{eq:WeylKac}
  \ch L(\lambda) = \frac{\sum_{w \in W} \sgn(w) e^{w(\lambda+\rho)-\rho}}{\prod_{\alpha \in \Delta^{+}} (1-e^{-\alpha})^{\mult \alpha}},
\end{align}
where $W$ is the Weyl group and $\rho$ is the Weyl vector. An important feature of the Weyl--Kac character formula is that it \emph{overcounts}, i.e.\ the sum is not sign-free. Finding a sign-free character formula is a difficult task, partial results may be obtained using the theory of vertex operator algebras and equating the resulting formulas with the Weyl--Kac character formula gives rise to certain partition identities. This was the basic strategy employed by Lepowsky and Wilson~\cite{Lepowsky,Lepowsky2} to derive the Rogers--Ramanujan identities from character formulas for level 3 modules over $A_{1}^{(1)}$. Subsequently, other partition identities were obtained (some of which were previously unknown to the combinatorics community) using different levels and algebras, see, for example, \cite{Capparelli,Meurman,Meurman2,Meurman3,Nandi,Primc1,PrimcSikic,Siladic}. More detail on the history of this interaction can be found in the introduction of~\cite{DK19}.

Let us state in detail a partition identity which was obtained by Meurman and Primc \cite{Meurman2} through the study of vertex operator algebras for higher level modules of $A_{1}^{(1)}$.
An integer partition can equivalently be defined in terms of frequencies, i.e.\ by a sequence $(f_1,f_2,f_3,\dots)$ where for every natural number $k$, $f_k$ denotes the number of appearances of $k$ in the partition. Meurman and Primc's result reads as follows (slightly reformulated):
\begin{theo}[Meurman--Primc 1999]
  \label{th:MP}
  Let $n$ and $i$ be non-negative integers such that $0 \leq i \leq n$.
  Let $M_{i,n}(m)$ denote the number of partitions of $m$ in two colours, plain and underlined, such that their frequencies $(f_1,f_2,f_3,\dots)$ satisfy, for all $k \geq 1$: $$f_{\underline{2k}} = 0,$$
  \begin{align*}
    f_{2k+1}+f_{2k}+f_{2k-1} &\leq n,\\
    f_{2k}+f_{\underline{2k-1}}+f_{2k-1} &\leq n,\\
    f_{\underline{2k+1}}+f_{2k+1}+f_{2k}&\leq n,\\
    f_{\underline{2k+1}}+f_{2k}+f_{\underline{2k-1}}&\leq n,
  \end{align*}
  and
  $$
  f_{\underline{1}} \leq i, \qquad f_1 \leq n-i.
  $$
  Then
  \begin{equation}
    \label{eq:MP1}
    \sum_{n\geq 0} M_{i,n}(m)q^m = \frac{(q^{i+1},q^{n-i+1},q^{n+2};q^{n+2})_\infty}{(q;q^2)_\infty(q;q)_\infty}.
  \end{equation}
\end{theo}
Note that up to the factor $1/(q;q^2)_{\infty}$, the product side of~\eqref{eq:MP1} is exactly the product side of the Andrews--Gordon and Bressoud identities.
Meurman and Primc's identity is a great example of a partition identity for which it seems highly unlikely that guessing the identity using only the combinatorics of integer partitions would have been possible, and where representation theory plays a key role in the shape of the difference conditions.

\medskip

An alternative approach to finding a sign-free character formula is provided by the theory of \emph{crystal bases}, a good reference text for which is~\cite{HK}. Said theory allows one to describe the character $\ch L(\lambda)$ in terms of a \emph{crystal} i.e.\ a directed graph $\B$ together with a weight function $\wta \colon \B \to P$ satisfying certain conditions (see, e.g.\ Defintion~4.5.1 in \cite{HK}). In particular, the crystal graph $\B(\lambda)$ of a crystal basis of $L(\lambda)$ will satisfy
\begin{align*}
  \ch L(\lambda) = \sum_{b \in \B(\lambda)} e^{\wta b}.
\end{align*}
Kashiwara~\cite{10.1215/S0012-7094-91-06321-0} proved the existence and uniqueness of crystal bases. Therefore, understanding the structure of the corresponding crystal graph automatically grants one a sign-free character formula. Kashiwara et al.\ ~\cite{KMN2} then provided a description of $\B(\lambda)$ in terms of certain particularly well-behaved crystals for the associated \emph{classical} algebra, called \emph{perfect crystals}. As before, equating the resulting crystal character formula with the Weyl--Kac formula can lead to a partition identity. This was first accomplished by Primc~\cite{Primc}, using level 1 modules over $A_{1}^{(1)}$ and $A_{2}^{(1)}$. This work was later generalised to level 1 modules over $A_{n}^{(1)}$ by the first and third authors in~\cite{DK19,DK19-2}, and generalised to treat other modules in \cite{DKmulti} (more on this in Section \ref{sec:ground}). In this paper, we study different perfect $A_{1}^{(1)}$ crystals of arbitrary level and provide new partition identities and sign-free character formulas.

\medskip
We now describe our main result. Let $n$ be a non-negative integer. Let $\mathcal{C}_n$ denote the set of $(n+1)$-coloured partitions $(\la_1,\dots,\la_s)$, where each part is a non-negative integer indexed by a colour taken from $\{c_0,c_1,\dots,c_n\}$, such that for all $1 \leq i \leq s-1$, $$\la_i-\la_{i+1}=|u_i - u_{i+1}|,$$
where for all $i\in \{1,\dots,s\}$, $\la_i$ has colour $c_{u_i}$. Similarly, let $\mathcal{C}_n^{\geq }$ denote the set of $(n+1)$-coloured partitions such that  $\la_i-\la_{i+1}\geq |u_i - u_{i+1}|.$

Our main result is a new family of partition identities which are companions (i.e.\ same infinite product but other difference conditions) to the Meurman--Primc identities.

\begin{theo}\label{th:mainidentity}
  Let $n$ and $i$ be non-negative integers such that $0 \leq i \leq n$.
  Let $C_{i,n}(m)$ be the number of $(n+1)$-coloured partitions of $m$ in $\mathcal{C}_n$ such that the last part is $0_{c_i}$ and the penultimate part has colour different from $c_i$. Let $C_{i,n}^{\geq}(m)$ be the number of $(n+1)$-coloured partitions of $m$ in $\mathcal{C}_n^{\geq}$ such that the last part is $0_{c_i}$ and the penultimate part is different from $0_{c_i}$. Then, we have the identities
  \begin{align}
    \label{eq:agcompanion}
    \sum_{n\geq 0} C_{i,n}(m)q^m &= \frac{(q^{i+1},q^{n-i+1},q^{n+2};q^{n+2})_\infty}{(q;q^2)_\infty(q;q)_\infty},\\
    \label{eq:agcompanionbis}
    \sum_{n\geq 0} C_{i,n}^{\geq}(m)q^m &= \frac{(q^{i+1},q^{n-i+1},q^{n+2};q^{n+2})_\infty}{(q;q^2)_\infty(q;q)^2_\infty}.
  \end{align}
\end{theo}

Just like for the Meurman--Primc identity, up to the factor $1/(q;q^2)_\infty$ (resp.\ $1/((q;q^2)_\infty (q;q)_{\infty}))$,
the product on the right-hand side of \eqref{eq:agcompanion} (resp.\ \eqref{eq:agcompanionbis}) is exactly the product side of the Andrews--Gordon and Bressoud identities. On the other hand, it is -- at least to the extent of our knowledge -- the first time in the literature that difference conditions involving absolute values arise in partition identities. It is also interesting to note that our difference conditions are very different from those of Theorem \ref{th:MP}, and arguably simpler. This highlights the fact that the approaches via vertex operator algebras and crystal bases, though both in the realm of representation theory, are very different and give rise to very different results.

To prove Theorem \ref{th:mainidentity}, we use energy matrices for level $n$ perfect crystals of $A_1^{(1)}$ and the theory of grounded partitions introduced by the first and third authors \cite{DK19-2} to express the characters of standard modules as generating functions for coloured partitions, which gives the left-hand side of \eqref{eq:agcompanion}. The right-hand side of \eqref{eq:agcompanion} comes from the principal specialisation of the Weyl--Kac character formula.

\medskip
Finally, to illustrate the effectiveness of (multi-)grounded partitions in the study of characters of standard modules we build upon work of the first and third authors on multi-grounded partitions \cite{DKmulti} and derive certain simple \textit{non-specialised} character formulas with manifestly positive coefficients for all level 2 standard modules of $A_1^{(1)}$. Although these formulas are not original, see Remark~\ref{rem:known-to-experts}, we find value in their proofs as presented here, as they are combinatorial in nature and may therefore prove more approachable for certain readers than pre-existing arguments.

Let $G=G(x_1, \dots, x_n)$ be a power series in $x_1, \dots, x_n$. For $k \leq n$, we denote by $\E_{x_1, \dots , x_k} (G)$ the sub-series of $G$ where we only keep the terms in which the total degree in $x_1, \dots , x_k$ is even.
If $G$ has only positive coefficients, then  for all $k$, $\E_{x_1, \dots , x_k} (G)$ also has positive coefficients and can be obtained easily from $G$ via the formula
\begin{equation}
  \label{eq:evennumberparts}
  \E_{x_1, \dots , x_k} (G)= \frac{1}{2} \Big(G(x_1, \dots, x_k, x_{k+1}, \dots , x_n) + G(-x_1, \dots, -x_k, x_{k+1}, \dots , x_n)\Big).
\end{equation}
\begin{prop}\label{prop:characterlevel2}
  Let $\Lambda_0, \Lambda_{1}$ be the fundamental weights and $\alpha_0, \alpha_1$ be the simple roots of $A_{1}^{(1)}.$
  Let $\delta= \alpha_0+\alpha_1$ be the null root. We have
  \begin{align}
    e^{-(\Ll_0+\Ll_1)}\ch L(\Ll_0+\Ll_1) &= (-e^{-\alpha_0},-e^{-\alpha_1},-e^{-\delta};e^{-\delta})_\infty, \label{eq:a11}
    \\e^{-2\Ll_0}\ch L(2\Ll_0) &= \frac{1}{2}\left[(-e^{-\frac{\delta}{2}+\alpha_1},-e^{-\frac{\delta}{2}-\alpha_1},-e^{-\frac{\delta}{2}};e^{-\delta})_\infty+(e^{-\frac{\delta}{2}+\alpha_1},e^{-\frac{\delta}{2}-\alpha_1},e^{-\frac{\delta}{2}};e^{-\delta})_\infty\right], \label{eq:a12}
    \\e^{-2\Ll_1}\ch L(2\Ll_1) &= \frac{1}{2}\left[(-e^{-\frac{3\delta}{2}+\alpha_1},-e^{\frac{\delta}{2}-\alpha_1},-e^{-\frac{\delta}{2}};e^{-\delta})_\infty+(e^{-\frac{3\delta}{2}+\alpha_1},e^{\frac{\delta}{2}-\alpha_1},e^{-\frac{\delta}{2}};e^{-\delta})_\infty\right]. \label{eq:a13}
  \end{align}
  Note that \eqref{eq:a12} and \eqref{eq:a13} are in the form of \eqref{eq:evennumberparts} and therefore have manifestly positive coefficients.
\end{prop}

\begin{rem}
  \label{rem:known-to-experts}
  These character formulas are not original. Indeed they can be derived by applying the `denominator' formula,
    \begin{align*}
      \sum_{w \in W} \sgn(w)e^{w(\rho)-\rho} = \prod_{\alpha \in \Delta_+} (1-e^{-\alpha})^{\mult\alpha},
    \end{align*}
    to the numerator of the Weyl--Kac character formula~\eqref{eq:WeylKac} with $ \lambda = \rho $.
In the case of $A_1^{(1)}$, the denominator formula is the same as Jacobi's triple product identity.
Equations \eqref{eq:a12} and \eqref{eq:a13} can be derived by applying this identity to \cite[(2.18b)]{MR1873994} and then combining the result with \cite[(2.17)]{MR1873994}.
\end{rem}

The paper is organised as follows. In Section \ref{sec:ground}, we recall the necessary background on (multi-)grounded partitions. In Section \ref{sec:crystal}, we study the energy matrices of level $n$ perfect crystals of $A_1^{(1)}$ and prove Theorem \ref{th:mainidentity}. Finally, in Section \ref{sec:level2}, we derive the non-specialised character formulas of Proposition \ref{prop:characterlevel2} and notice an intriguing connection with Capparelli's identity.

\section{Background on perfect crystals and multi-grounded partitions}
\label{sec:ground}
We start by briefly recalling the theory of perfect crystals. Let $\gf$ be an affine Kac--Moody algebra with simple positive roots $\alpha_{0},  \dots,  \alpha_{n}$ and with null root $\delta = d_{0}\alpha_{0} + \cdots + d_{n}\alpha_{n} $. For $\lambda \in \bar P^{+}$, let $\B(\lambda)$ the crystal graph of a crystal basis of $L(\lambda)$. For an integer level $\ell \geq 1$ and a weight $\lambda \in \bar P^+_\ell$, Kashiwara et al.\ ~\cite[Section 1.4]{KMN2} define the notion of a \emph{perfect crystal $\B$ of level $\ell$}, an \emph{energy function} $H\colon \B \otimes \B \to \mathbb{Z}$, and a particular element
$$
{\p}_\lambda = \bigl(g_k)_{k=0}^\infty =  \cdots \ot g_{k+1} \ot g_k \ot \cdots \ot g_1 \ot g_0 \in \B^{\infty},
$$
called the \emph{ground state path of weight} $\lambda$ (see Section 1.4 of their paper for precise definitions). From this they consider all elements of the form
$$\p = (p_k)_{k=0}^\infty =  \cdots \ot p_{k+1} \ot p_k \ot \cdots \ot p_1 \ot p_0 \in \B^{\infty},$$
which satisfy $p_k = g_k$ for large enough $k$. Such elements are called $\lambda$-paths; their collective set is denoted $\mathcal{P}(\lambda)$. The crystal $\B(\lambda)$ can then be realised on the set of $\lambda$-paths, in particular the affine weight function is given by the following theorem.
\begin{theo}[(KMN)$^2$ crystal base character formula \cite{KMN2a}]
  \label{theorem:wtchar}
  Let $\lambda \in \bar P^+_{\ell}$, let $H$ be an energy function on $\B \ot \B$, let $ \mathcal{P}(\lambda)$ be the set of $\lambda$-paths, and let $\p = (p_k)_{k=0}^\infty \in \mathcal P(\lambda)$.
  Then the weight of $\p$ is given by the following expressions:
  \begin{align*}
    \wta \p &= \lambda + \sum_{k=0}^\infty \left(\wt p_k -\wt g_k\right)  - \frac{\delta}{d_0}\sum_{k=0}^\infty (k+1)\Big(H(p_{k+1} \ot p_k) - H(g_{k+1}\ot g_k)\Big),   \\
            &= \lambda + \sum_{k=0}^\infty \left(\left(\wt p_k -\wt g_k\right)  -  \frac{\delta}{d_0}\sum_{\ell=k}^\infty (H(p_{\ell+1} \ot p_\ell) - H(g_{\ell+1}\ot g_\ell))\right)\! ,
  \end{align*}
  where $\wt$ is the weight function of $\B$. As $\mathcal{P}(\lambda) \cong \B(\lambda)$, we have the character formula
  \begin{align*}
    \ch L(\lambda) &= \sum_{\p \in \mathcal P(\lambda)}  e^{\wta \p}. %
  \end{align*}
\end{theo}

In \cite{DK19-2}, the first and third authors used bijections to transform the (KMN)$^2$ crystal base character formula into a formula expressing characters as generating functions for so-called ``grounded partitions'', in the case where the ground state path of the module considered was constant. In \cite{DKmulti}, they generalised their theory to treat the cases of all ground state paths, thereby introducing multi-grounded partitions. We now outline this latter more general theory, which contains the former as a particular case.

First, recall the definition of multi-grounded partitions.

\begin{deff}
  Let $\C$ be a set of colours, and let $\Z_{\C} = \{k_c : k \in \Z, c \in \C\}$ be the set of integers coloured with the colours of $\C$.
  Let $\succ$ be a binary relation defined on $\Z_{\C}$.
  A \textit{generalised coloured partition} with relation $\succ$ is a finite sequence $(\pi_0,\dots,\pi_s)$ of coloured integers, such that for all $i \in \{0, \dots, s-1\},$ $\pi_i\succ \pi_{i+1}.$
\end{deff}

In the following, if $\pi=(\pi_0,\dots,\pi_s)$ is a generalised coloured partition, then $c(\pi_i) \in \C$ denotes the colour of the part $\pi_i$. The quantity $|\pi|=\pi_0+\cdots+\pi_s$ is the weight of $\pi$, and $C(\pi) = c(\pi_0)\cdots c(\pi_s)$ is its colour sequence.

While multi-grounded partitions can be defined purely combinatorially, in practice we will use them with a colour set which is indexed by the vertices of a perfect crystal $\B$. Thus our notation for the colours' indices reflects that correspondence and can be readily used to set equalities as in the statement of Theorem \ref{th:char}.

\begin{deff}\label{deff:multiground}
  Let $\Co$ be a set of colors, $\Z_{\Co}$ the set of integers coloured with colours in $\C$, and $\succ$ a binary relation defined on $\Z_{\Co}$. Suppose that there exist some colors $c_{g_0},\dots,c_{g_{t-1}}$ in $\Co$
  and \textit{unique} coloured integers $u_{c_{g_0}}^{(0)},\dots, u_{c_{g_{t-1}}}^{(t-1)}$ such that
  \begin{align*}
    &u^{(0)}+\cdots+u^{(t-1)}=0,  \\ %
    &u_{c_{g_0}}^{(0)}\succ u_{c_{g_1}}^{(1)}\succ \dots\succ u_{c_{g_{t-1}}}^{(t-1)}\succ u_{c_{g_0}}^{(0)}.   \label{eq:cyclgrounds}
  \end{align*}
  Then a \textit{multi-grounded partition} with colours $\C$, ground $c_{g_0},\dots,c_{g_{t-1}}$ and relation $\succ$ is a non-empty generalised coloured partition $\pi = (\pi_0,\dots,\pi_{s-1},u_{c_{g_0}}^{(0)},\dots, u_{c_{g_{t-1}}}^{(t-1)})$ with relation $\succ$, such that
  $(\pi_{s-t},\dots,\pi_{s-1})\neq (u_{c_{g_0}}^{(0)},\dots, u_{c_{g_{t-1}}}^{(t-1)})$ in terms of coloured integers.
\end{deff}
\noindent We denote by $\Psucm$ the set of multi-grounded partitions with ground $c_{g_0},\dots,c_{g_{t-1}}$ and relation $\succ$. We do not make the set of colours $\C$ explicit in the notation as it should be clear from the context.

Grounded partitions, which were introduced in \cite{DK19-2}, are a particular case of multi-grounded partitions where the ground is reduced to one colour $c_{g_0}$ (and therefore the smallest part has to be $0_{c_{g_0}}$).

\medskip
We now recall the connection between multi-grounded partitions, ground state paths, perfect crystals and character formulas.

For $\ell\geq 1$, let $\B$ be a perfect crystal of level $\ell$ and let $\lambda \in \bar P_{\ell}^+$ be a level $\ell$ dominant classical weight with ground state path $\p_{\lambda}=(g_k)_{k\geq 0}$ (which is always periodic). Let us set $t$ to be the \textit{period} of the ground state path, i.e.\ the smallest non-negative integer $k$ such that $g_k=g_0$.
Let $H$ be an energy function on $\B\ot\B$. Since $\B\ot\B$ is connected, $H$ is uniquely determined by fixing its value on a particular $\tilde{b} \ot \tilde{b}' \in \B \ot \B$.

We now define the function $H_{\lambda}$, for all $b, b' \in \B\ot\B$, by
\begin{equation}\label{eq:Hlamb}
  H_{\lambda} (b\ot b') := H(b\ot b')-\frac{1}{t}\sum_{k=0}^{t-1} H(g_{k+1}\ot g_k)\,.
\end{equation}
Thus we have
\begin{equation*}
  \sum_{k=0}^{t-1} H_{\lambda}(g_{k+1}\ot g_k)=0.
\end{equation*}
The function $H_{\lambda}$ satisfies all the properties of energy functions, except that it does not have integer values unless $t$ divides $\sum_{k=0}^{t-1} H(g_{k+1}\ot g_k).$
With this new notation, we can rewrite the (KMN)$^2$ formula for the weight of a $\lambda$-path in the following way.
Let $m\geq 0$ and $\p = (p_k)_{k=0}^\infty \in \mathcal P(\lambda)$ such that $p_{m't+i}=g_i$ for $m'\geq m$ and $i\in \{0,\dots,t-1\}$. We have

\begin{equation}\label{eq:withlambda}
  \wta \p = \lambda +\sum_{k=0}^{mt-1} \wt p_k  +\frac{m\delta}{d_0}\sum_{k=0}^{t-1} (k+1) H_{\lambda}(g_{k+1}\ot g_k) -\frac{\delta}{d_0}\sum_{k=0}^{mt-1} (k+1)H_{\lambda}(p_{k+1} \ot p_k).
\end{equation}

Note that for any energy function $H$, we always have
\[\sum_{k=0}^{t-1} (k+1) H_{\lambda}(g_{k+1}\ot g_k)= \sum_{k=0}^{t-1} (k+1) H(g_{k+1}\ot g_k) -\frac{t+1}{2}\sum_{k=0}^{t-1} H(g_{k+1}\ot g_k) \in \frac{1}{2}\Z  .\]
The quantity above is an integer as soon as $t$ is odd, and is equal to $0$ when $t=1$. Thus we can choose a suitable divisor $D$ of $2t$ such that
$DH_{\lambda}(\B\ot\B)\subset \Z$ and $\frac{1}{t}\sum_{k=0}^{t-1} (k+1) DH_{\lambda}(g_{k+1}\ot g_k) \in \Z$. For the entirety of this paper, $D$ always denotes such an integer.

Let us now consider the set of colours $\Co_{\B}$ indexed by $\B$, and let us define the relations $\gtrdot$  and $\gg$ on $\Z_{\Co_{\B}}$ by
\begin{align}\label{eq:DHLeg}
  k_{c_b}\gtrdot k'_{c_{b'}} &\Longleftrightarrow k-k'= DH_\lambda(b'\ot b), \\
  \label{eq:DHLin}
  k_{c_b}\gg k'_{c_{b'}} &\Longleftrightarrow k-k'\geq DH_\lambda(b'\ot b).
\end{align}
We can define multi-grounded partitions associated with these relations, as can be seen in the next proposition.

\begin{prop}[\cite{DKmulti}]
  \label{prop:multigroundedparts}
  The set $\Pppm$ (resp. $\Ppm$) of multi-grounded partitions with ground $c_{g_0},\dots,c_{g_{t-1}}$ and relation $\gtrdot$ (resp. $\gg$) is the set of non-empty generalised coloured partitions $\pi = (\pi_0,\dots,\pi_{s-1},u_{c_{g_0}}^{(0)},\dots, u_{c_{g_{t-1}}}^{(t-1)})$ with relation $\gtrdot$ (resp. $\gg$), such that
  $(\pi_{s-t},\dots,\pi_{s-1})\neq (u_{c_{g_0}}^{(0)},\dots, u_{c_{g_{t-1}}}^{(t-1)})$, and for all $k \in \{0, \dots , t-1\},$
  \begin{equation*}%
    u^{(k)} = -\frac{1}{t} \sum_{\ell=0}^{t-1} (\ell+1) DH_{\lambda}(g_{\ell+1}\ot g_\ell) + \sum_{\ell=k}^{t-1}DH_{\lambda}(g_{\ell+1}\ot g_\ell).
  \end{equation*}
\end{prop}

In order to give a more general connection between multi-grounded partitions and character formulas, we use some further sets of multi-grounded partitions.
For any positive integer $d$, let
$^d\Ppm$ denote the set of multi-grounded partitions $\pi=(\pi_0,\dots,\pi_{s-1},u_{c_{g_0}}^{(0)},\dots, u_{c_{g_{t-1}}}^{(t-1)})$ of $\Ppm$ such that for all $k\in\{0,\dots,s-1\}$,
\begin{equation*}
  \pi_k-\pi_{k+1} - DH_{\lambda}(p_{k+1}\ot p_{k}) \in d\Z_{\geq 0},
\end{equation*}
where $c(\pi_k)=c_{p_k}$ and $\pi_s=u_{c_{g_0}}^{(0)}$.

Finally, let $_t^d\Ppm$ (resp. $_t\Ppm$, $_t\Pppm$) denote the set of multi-grounded partitions of $^d\Ppm$ (resp. $\Ppm$, $\Pppm$) whose number of parts is divisible by $t$.

Now the character formulas connecting perfect crystals to multi-grounded partitions can be stated as follows.

\begin{theo}[Dousse--Konan 2021]
  \label{th:char}
  Setting $q=e^{-\delta/(d_0D)}$ and $c_b=e^{\wt b}$ for all $b\in \B$, we have $c_{g_0}\cdots c_{g_{t-1}}=1$, and the character of the irreducible highest weight module
  $L(\lambda)$ is given by the following expressions:
  \begin{align*}
    \sum_{\mu\in _t\Pppm} C(\pi)q^{|\pi|} &= e^{-\lambda}\ch L(\lambda),\\
    \sum_{\pi\in \, _t^d\Ppm} C(\pi)q^{|\pi|} &= \frac{e^{-\lambda}\ch L(\lambda)}{(q^d;q^d)_{\infty}}.
  \end{align*}
\end{theo}

This theorem will be used directly to prove Proposition \ref{prop:characterlevel2} in Section \ref{sec:level2} and the proof of Theorem \ref{th:mainidentity} in Section \ref{sec:crystal} relies on a variant of these ideas.

\section{Perfect crystals for standard $A_1^{(1)}$ modules of level $n$}
\label{sec:crystal}
The level $n$ perfect crystal $\B_n$ of $A_1^{(1)}$ is shown on Figure \ref{fig:crystalleveln},
and for all $i$, the weights and null root are given, respectively, by
\begin{equation}
  \label{eq:star}
  \begin{aligned}
    \wt b_i &= (2i -n)\Lambda_0 + (n-2i)\Lambda_1 = \left(\frac{n}{2}-i\right)\alpha_1, \\
    \delta &=\alpha_0+\alpha_1,
  \end{aligned}
\end{equation}
where $\alpha_0$ and $\alpha_1$ are the simple roots ~\cite[Example 10.5.2]{HK}.

\begin{figure}[ht]
  \begin{center}
    \includegraphics[width=0.7\textwidth]{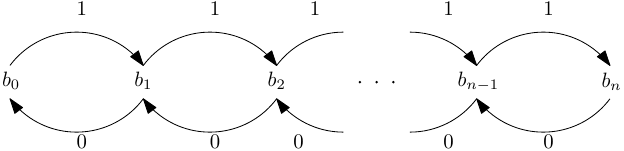}
  \end{center}
  \caption{The level $n$ perfect crystal $\B_n$ of $A_1^{(1)}$}
  \label{fig:crystalleveln}
\end{figure}

There are $n+1$ standard $A_1^{(1)}$ modules of level $n$, namely whose highest weight is of the form
$$\Lambda_{i,n}:= i \Lambda_0 + (n-i) \Lambda_1,$$
for $i \in \{0,\dots,n\}$.
Their ground state paths are
$$\p_i := \p_{\Lambda_{i,n}} = \cdots \ot b_{n-i} \ot b_i \ot b_{n-i} \ot b_i.$$

When $n$ is even, then there is a constant ground state path for $i=n/2$. Otherwise, all the ground state paths have period $2$. By \cite[Lemma 4.6.2]{KMN2}, $\B_n\ot\B_n$ is also a level $n$ perfect crystal, the corresponding energy function is given by the matrix below (where $H_n (b_i \ot b_j) $ is given in column $i$, row $j$):

\[
  H_n=
  \bordermatrix{
    \text{}&b_0& b_1 &b_2 &\cdots &b_{n-2}& b_{n-1} &b_n
    \cr b_0&n&n&n&\cdots&n&n &n
    \cr b_1&n-1&n-1&n-1&\cdots&n-1 &n-1&n
    \cr b_2&n-2&n-2&n-2&\cdots&n-2 &n-1 &n
    \cr \vdots&\vdots&\vdots&\vdots&\cdots &\vdots &\vdots &\vdots
    \cr b_{n-2}&2&2&2&\cdots&n-2 &n-1 &n
    \cr b_{n-1}&1&1&2&\cdots&n-2 &n-1 &n
    \cr b_n&0&1&2&\cdots&n-2 &n-1 &n
  } .
\]

We can rewrite this more concisely, for all $i,j \in \{0, \dots, n\}$, as
\begin{equation*}
  H_n (b_i \ot b_j) = \max(i,n-j).
\end{equation*}

While, to prove Theorem \ref{th:mainidentity}, we will not exactly use the theory of multi-grounded partitions explained in Section \ref{sec:ground}, we still need to consider a (pseudo) energy function $H_\lambda$ such that the sum of its values on the ground state path of $\lambda$ is $0$. For this, we use \eqref{eq:Hlamb}.

Here, for all the ground state paths $\p_i$, we have
$$H_n(b_{n-i}\ot b_i) + H_n(b_i\ot b_{n-i})=n.$$

Thus, for all standard modules $\lambda$ of level $n$, we have
$$H_{\lambda} = H_n - \frac{n}{2},$$
or in other words, for all $i,j \in \{0, \dots, n\}$:
\begin{equation*}
  H_\lambda (b_i \ot b_j) = \max(i-\frac{n}{2},\frac{n}{2}-j),
\end{equation*}

\[
  H_\lambda=
  \bordermatrix{
    \text{}&b_0& b_1 &b_2 &\cdots &b_{n-2}& b_{n-1} &b_n
    \cr \vspace{0.3em} b_0&\frac{n}{2}&\frac{n}{2}&\frac{n}{2}&\cdots&\frac{n}{2}&\frac{n}{2} &\frac{n}{2}
    \cr \vspace{0.3em} b_1&\frac{n}{2}-1&\frac{n}{2}-1&\frac{n}{2}-1&\cdots&\frac{n}{2}-1 &\frac{n}{2}-1&\frac{n}{2}
    \cr \vspace{0.3em}b_2&\frac{n}{2}-2&\frac{n}{2}-2&\frac{n}{2}-2&\cdots&\frac{n}{2}-2 &\frac{n}{2}-1 &\frac{n}{2}
    \cr \vspace{0.3em} \vdots&\vdots&\vdots&\vdots&\cdots &\vdots &\vdots &\vdots
    \cr \vspace{0.3em} b_{n-2}&-\frac{n}{2}+2&-\frac{n}{2}+2&-\frac{n}{2}+2&\cdots&\frac{n}{2}-2 &\frac{n}{2}-1 &\frac{n}{2}
    \cr \vspace{0.3em} b_{n-1}&-\frac{n}{2}+1&-\frac{n}{2}+1&-\frac{n}{2}+2&\cdots&\frac{n}{2}-2 &\frac{n}{2}-1 &\frac{n}{2}
    \cr \vspace{0.3em} b_n&-\frac{n}{2}&-\frac{n}{2}+1&-\frac{n}{2}+2&\cdots&\frac{n}{2}-2 &\frac{n}{2}-1 &\frac{n}{2}
  } .
\]

When $n$ is even, $H_{\lambda}$ is really an energy function. Otherwise, it has all the properties of an energy function, except that it has half-integral values.

To obtain the very simple difference conditions of Theorem \ref{th:mainidentity}, we perform yet another rewriting of $H_{\lambda}$.
For $0\leq i,j\leq n$, we have
\begin{align}
  H_\lambda (b_i \ot b_j) &= \frac{1}{2}\max(j+i-n,n-i-j) +\frac{1}{2}(i-j) \,, \nonumber
  \\
                          &= \frac{1}{2}|n-i-j|+\frac{1}{2}(i-j). \label{eq2}
\end{align}
Thus
\begin{equation}
  \label{eq1}
  H_{\lambda}(b_{n-i} \ot b_{i}) + 2 H_{\lambda}(b_{i} \ot b_{n-i}) = H_{\lambda}(b_{i} \ot b_{n-i}) = i-\frac{n}{2}.
\end{equation}

Let us consider the set of colours $\C=\{c_0,\dots,c_n\}$ and define the difference condition $\Delta(c_a,c_b)=|a-b|$. Denote by $\Pp_{i,n}$ the set of grounded partitions with colours $\C$, ground $c_i$ and relation $\gtrdot$ defined by:
\begin{equation*}%
  k_{c_a} \gtrdot l_{c_b} \text{ if and only if } k-l= \Delta(c_a,c_b).
\end{equation*}

We shall now prove a key proposition which relates the partitions in $\Pp_{i,n}$ and the character of $L(\Lambda_{i,n})$.

\begin{prop}\label{prop:main}
Let $x=\exp(-\frac{\alpha_0+\alpha_1}{2})$, $c_k= \exp(k\cdot\frac{\alpha_1-\alpha_0}{2})$, and
  $$C'(\pi) = \prod_{a=0}^{s} (c_{\pi_a})^{-1^{a+1}}
  \times \begin{cases}
    1  & \text{if} \ s \ \text{is odd}\\
    c_i & \text{if} \ s \ \text{is even.}\\
  \end{cases}
  $$
  Then there exists a bijection $\Phi$ between $\Pp(\Lambda_{i,n})$ and $\Pp_{i,n}$ such that, if $\Phi(\p)=\pi =(\pi_0,\dots,\pi_s,0_{c_i})$, then
  \begin{equation}\label{eq:weightsize}
    \exp(\wta \p -\Lambda_{i,n}) = C'(\pi)q^{|\pi|}.
  \end{equation}
  Hence,
  \begin{equation}\label{eq:charmain}
    \sum_{\pi \in \Pp_{i,n}} C'(\pi) x^{|\pi|} = \exp(-\Lambda_{i,n})\ch L(\Lambda_{i,n}).
  \end{equation}
\end{prop}
\begin{proof}
  In this proof, we set $\la=\Lambda_{i,n}$.
  Let $\p=(b_{i_k})_{k\geq 0}$, with $i_k\in \{0,\dots,n\}$ for all $k$, be a $\la$-path, and let $m$ be the unique non-negative integer such that $(i_{2m-2},i_{2m-1})\neq(i,n-i)$ and $(i_{2m'},i_{2m'+1})=(i,n-i)$ for all $m'\geq m$.
  We compute $\wta \p$, defined in \eqref{eq:withlambda}, in terms of the roots $\alpha_0$ and $\alpha_1$. By \eqref{eq:star}, \eqref{eq2} and \eqref{eq1}, we have
  \begin{align*}
    \wta \p&= \lambda +\sum_{k=0}^{2m-1} \left(\frac{n}{2}-i_k\right)\alpha_1 +m(\alpha_0+\alpha_1)\left(i-\frac{n}{2}\right)+\frac{\alpha_0+\alpha_1}{2}\sum_{k=0}^{2m-1} (k+1)(i_k-i_{k+1})\\
           &\qquad\qquad -\frac{\alpha_0+\alpha_1}{2}\sum_{k=0}^{2m-1}(k+1)|n-i_{k+1}-i_k|.
  \end{align*}
  Using the fact that
  $$\sum_{k=0}^{2m-1} (k+1)(i_k-i_{k+1}) = \sum_{k=0}^{2m-1} i_k -2m i_{2m},$$
  and that $i_{2m}=i$ because $p_{2m}=b_i$, we simplify and obtain
  \begin{align*}
    \wta \p&= \lambda +\sum_{k=0}^{2m-1} \left(\frac{n}{2}-i_k\right)\alpha_1 +m(\alpha_0+\alpha_1)\left(i-\frac{n}{2}\right) -m(\alpha_0+\alpha_1)i+\frac{\alpha_0+\alpha_1}{2}\sum_{k=0}^{2m-1}i_k\\
           &\qquad\qquad -\frac{\alpha_0+\alpha_1}{2}\sum_{k=0}^{2m-1}(k+1)|n-i_{k+1}-i_k|\\
           &= \lambda +\sum_{k=0}^{2m-1} \left[ \left(\frac{n}{2}-i_k\right)\alpha_1-\frac{\alpha_0+\alpha_1}{2}\left(\frac{n}{2}-i_k\right) \right] -\frac{\alpha_0+\alpha_1}{2}\sum_{k=0}^{2m-1}(k+1)|n-i_{k+1}-i_k|\\
           &= \lambda +\sum_{k=0}^{2m-1} \left(\frac{n}{2}-i_k\right)\frac{\alpha_1-\alpha_0}{2} -\frac{\alpha_0+\alpha_1}{2}\sum_{k=0}^{2m-1}(k+1)|n-i_{k+1}-i_k|.
  \end{align*}
  By setting $j_{2k+1}=n-i_{2k+1}$ and $j_{2k}=i_{2k}$, this reduces to
  \begin{align}
    \wta \p -\lambda&= \sum_{k=0}^{2m-1} (-1)^{k+1} j_k \cdot \frac{\alpha_1-\alpha_0}{2} -\frac{\alpha_0+\alpha_1}{2}\sum_{k=0}^{2m-1}(k+1)|j_{k+1}-j_k|\nonumber\\
                    &= \sum_{k=0}^{2m-1} \left((-1)^{k+1} j_k \cdot \frac{\alpha_1-\alpha_0}{2} -\frac{\alpha_0+\alpha_1}{2}\sum_{l=k}^{2m-1}|j_{l+1}-j_l| \right)\!.
                      \label{eq:newground}
  \end{align}
  Observe that when $i_{2m-1}=n-i$, we have $|j_{2m-1}-j_{2m}|=|i-i|=0$. Furthermore, $(i_{2m-2},i_{2m-1})\neq (i,n-i)$ if and only if $(j_{2m-2},j_{2m-1})\neq (i,i)$. Thus, for $k\in\{0,\dots, 2m-1\}$, we set
  $$\pi_k=\left(\sum_{l=k}^{2m-1}|j_k-j_{k+1}|\right)_{c_{j_k}},$$
  and
  $$
  \Phi(\pi)= \begin{cases}
    (\pi_0,\dots,\pi_{2m-2},0_{c_i}) & \text{if} \ \ j_{2m-1}=i\\
    (\pi_0,\dots,\pi_{2m-1},0_{c_i}) & \text{if} \ \ j_{2m-1}\neq i.
  \end{cases}
  $$
  Hence, for all $0\leq k\leq 2m-2$, we have
  $\pi_{k}-\pi_{k+1}=|j_k-j_{k+1}|$, i.e.\ $\pi_k\gtrdot \pi_{k+1}.$
  Finally, when $j_{2m-1}=i$, we indeed have $j_{2m-2}\neq i$, so $\pi$ is a grounded partition with ground $c_i$. When $j_{2m-1}\neq i$, $\pi$ is also a grounded partition with ground $c_i$, as $\pi_{2m-1}=|j_{2m-1}-j_{2m}|=|j_{2m-1}-i|$ and then $\pi_{2m-1}\gtrdot 0_{c_i}$. Therefore, $\pi$ always belongs to $\Pp_{i,n}$, and \eqref{eq:newground} becomes
  $$\wta \p -\Lambda_{i,n} = \frac{\alpha_1-\alpha_0}{2} \left(\sum_{k=0}^{2m-1} (-1)^{k+1} j_k\right)  -\frac{\alpha_0+\alpha_1}{2} |\pi|,$$
  where $|\pi|$ denotes the size of $\pi$.
  \\\\
  Let us now give the inverse bijection. Let $\pi = (\pi_0,\dots,\pi_{s},0_{c_i}) \in \Pp_{i,n}$, with $c(\pi_k)=j_k$ for all $0\leq k\leq s+1$. In particular, $j_{s+1}=i$. Then set $m=\left\lceil \frac{s+1}{2}\right\rceil$,
  $$
  (i_{2k},i_{2k+1}) = \begin{cases}
    (j_{2k},n-j_{2k+1} ) & \text{if}\ k\leq m-1\\
    (i,n-i) & \text{if}\ k\geq m,\\
  \end{cases}
  $$
  and $\Phi^{-1}(\pi)= \p = (b_{i_k})_{k\geq 0}$. Since $j_s\neq i$ and $s\in \{2m-2,2m-1\}$, we then have  $(j_{2m-2},j_{2m-1})\neq (i,i)$, so that $m$ is the unique non-negative integer satisfying $(i_{2m-2},i_{2m-1})\neq(i,n-i)$ and $(i_{2m'},i_{2m'+1})=(i,n-i)$ for all $m'\geq m$. Hence, $\p \in \Pp(\Lambda_{i,n})$. Moreover, as $j_{s+1}= i$ and $s+1\in \{2m-1,2m\}$, setting $j_{2m}=i$ yields
  \begin{align*}
  \pi_k &= \sum_{l=k}^s |j_l-j_{l+1}| = \sum_{l=k}^{2m-1} |j_l-j_{l+1}|,\\
  |\pi| &=\sum_{k=0}^s \pi_k = \sum_{k=0}^{2m-1} \sum_{l=k}^{2m-1} |j_l-j_{l+1}|,
  \end{align*}
  and applying \eqref{eq:newground} to $\p$ gives
  $$\wta \p -\Lambda_{i,n} = \frac{\alpha_1-\alpha_0}{2} \left(\sum_{k=0}^{2m-1} (-1)^{k+1} j_k\right)  -\frac{\alpha_0+\alpha_1}{2} |\pi|.$$\\\\
  The uniqueness of the relations between $s$ and $m$, as well as the definition of the sequences $(i_k)_{k\geq 2m-1}$ and $(j_{k})_{k=0}^{s+1}$, imply that $\Phi$ and $\Phi^{-1}$ are inverses of each other. The relation \eqref{eq:weightsize} comes from the fact that, for $\pi = \Phi(\p)$,
  $$\wta \p -\Lambda_{i,n} = \frac{\alpha_1-\alpha_0}{2} \left(\sum_{k=0}^{2m-1} (-1)^{k+1} j_k\right)  -\frac{\alpha_0+\alpha_1}{2} |\pi|.$$
\end{proof}
We now consider the set $\Pp_{i,n}^{\geq}$ of grounded partitions with colours $\C$, ground $c_i$ and relation $\gg$ defined by:
\begin{equation*}%
  k_{c_a}\gg l_{c_b} \text{ if and only if } k-l\geq  \Delta(c_a,c_b)\,\cdot
\end{equation*}
Denote by $\Pp$ the set of classical partitions.
\begin{prop}\label{prop:mainbis}
  There exists a weight-preserving bijection $\Psi$ between $\Pp_{i,n}^{\geq}$ and $\Pp_{i,n}\times \Pp$, i.e. if $\Psi(\pi)=(\mu,\nu)$ then $|\pi|=|\mu|+|\nu|$.
\end{prop}
\begin{proof}
  Here we only describe the bijection $\Psi$. The proof of its well-definedness is analogous the proof of Proposition 3.6 in \cite{DK19-2}. Let $\pi=(\pi_0,\dots,\pi_{s-1},0_{c_i})\in \Pp_{i,n}^{\geq}$, and set
  $$r=\max \{k \in \{0, \dots , s\}:\,c(\pi_{k-1})\neq c_i\}.$$
  We set $\mu:=(\mu_0,\dots,\mu_{r-1},0_{c_i})$, where the part $\mu_k$ is coloured $c(\pi_k)$ and has size
  $$\sum_{l=k}^{r-1}\Delta(c(\pi_l),c(\pi_{l+1}))\,\cdot$$
  We now build $\nu=(\nu_0,\dots,\nu_{t-1})$. If $r<s$, then $t=s$ and
  $$\nu_k := \begin{cases}
    \pi_k-\mu_k & \text{for} \ k\in \{0,\dots,r-1\},\\
    \pi_k & \text{for} \ k\in \{r,\dots,s-1\}.\\
  \end{cases}
  $$
  Otherwise $r=s$, in which case we set
  $$t=\min\{k \in \{0, \dots , s\}: \pi_k=\mu_k\},$$
  and $\nu_k:=\pi_k-\mu_k$ for all $k\in \{0,\dots,t-1\}$.
\end{proof}

We are now ready to prove our main theorem.
\begin{proof}[Proof of Theorem \ref{th:mainidentity}]
  First observe that the quantity $C_{i,n}(m)$ counts the partitions of $m$ in $\Pp_{i,n}$, and thus
  $$\sum_{m\geq 0} C_{i,n}(m)q^m= \sum_{\pi \in \Pp_{i,n}} q^{|\pi|}.$$
  The principal specialisation of the character, which consists in performing the transformations $e^{-\alpha_0}, e^{-\alpha_1} \mapsto q$, is denoted by $\mathds{1}$. In Proposition \ref{prop:main}, it implies the transformations $x \mapsto q$ and $c_k \mapsto 1$.
  Hence, we obtain that $C'(\pi)=1$ for all $\pi \in \Pp_{i,n}$, and the equality \eqref{eq:charmain} becomes
  $$\sum_{m\geq 0} C_{i,n}(m)q^m= \mathds{1}(\exp(-\Lambda_{i,n})\ch L(\Lambda_{i,n})).$$
  Finally, the right-hand side of \eqref{eq:agcompanion} follows from the principal specialisation of the Weyl--Kac character formula \eqref{eq:WeylKac} for the type $A^{(1)}_{1}$ given in \cite{Le-Mi1}:
  $$\mathds{1}(\exp(-\Lambda_{i,n})\ch L(\Lambda_{i,n})) = \prod_{\alpha \in \Delta_{+}^{\vee}} \left(\frac{1-q^{\langle\Ll_{i,n}+\rho,\alpha\rangle}}{1-q^{\langle\rho,\alpha\rangle}}\right)^{\text{mult }\alpha} = \frac{(q^{i+1},q^{n-i+1},q^{n+2};q^{n+2})_{\infty}}{(q,q,q^2;q^2)_{\infty}},$$
  where the second equality is due to the fact that, in our $A_{1}^{(1)}$ case, we have
  $$
  \begin{cases}
    \rho=\Ll_0+\Ll_1,\\
    \Delta_{+}^{\vee} = \{k\alpha_0^\vee+(k-1)\alpha_1^\vee, (k-1)\alpha_0^\vee+k\alpha_1^\vee,k\alpha_0^\vee+k\alpha_1^\vee\,:\, k \in \Z_{\geq 1}\},\\
    \text{mult } \alpha = 1 \ \ \text{for all }\alpha \in \Delta_{+}^{\vee},\\
    \langle \Ll_u,\alpha_v^\vee\rangle = \chi(u=v),
  \end{cases}
  $$
  where $\Delta_{+}^{\vee}$ is the set of positive coroots (in $A_n^{(1)}$, the coroots can be identified with the roots).
  To prove \eqref{eq:agcompanionbis}, we similarly remark that $C_{i,n}^{\geq}(m)$ counts the partitions of $m$ in $\Pp_{i,n}^{\geq}$, so that by Proposition \ref{prop:mainbis},
  \begin{align*}
    \sum_{m\geq 0} C_{i,n}^{\geq}(m)q^m&= \sum_{\pi \in \Pp_{i,n}^{\geq}} q^{|\pi|}\\
                                       &=\frac{1}{(q;q)_{\infty}} \sum_{\pi \in \Pp_{i,n}} q^{|\pi|}.
        \qedhere                               
  \end{align*}
\end{proof}

\section{Non-specialised character formulas for level 2 standard modules}
\label{sec:level2}
In this section we pay particular attention to the level $2$ standard modules of $A_1^{(1)}$ and show how the relationship between (multi-)grounded partitions and the characters of standard modules may be used to rederive the non-specialised character formulas of Proposition \ref{prop:characterlevel2} using purely combinatorial techniques.

\begin{figure}[h]
  \label{fig:crystallevel2}
  \begin{center}
    \includegraphics[width=0.3\textwidth]{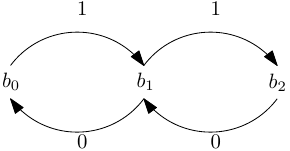}
  \end{center}
  \caption{The level $2$ perfect crystal of $A_1^{(1)}$}
\end{figure}

The  perfect crystal $\mathcal{B}_2$ of $A_1^{(1)}$ level $2$ can be seen in Figure \ref{fig:crystallevel2}, and the an energy matrix is given by

\[
  H_2=
  \bordermatrix{
    \text{}&b_0& b_1 &b_2
    \cr b_0&2&2&2
    \cr b_1&1&1&2
    \cr b_2&0&1&2}.
\]

Here, we have three highest weights of level $2$:
\begin{itemize}
\item $\Lambda_0 + \Lambda_1$, with ground state path $\cdots \ot b_1 \ot b_1$,
\item $2 \Lambda_0$, with ground state path $\cdots \ot b_0 \ot b_2 \ot b_0 \ot b_2$,
\item $2 \Lambda_1$, with ground state path $\cdots \ot b_2 \ot b_0 \ot b_2 \ot b_0$.
\end{itemize}

Note that $H_2$ is exactly the non-dilated matrix of difference conditions for Capparelli's identity reformulated by the first author \cite{DousseCapaPrimc}. Unfortunately, this does not allow us to directly connect Capparelli's identity to perfect crystals, as the connection is actually made with energy functions whose sum is zero on the ground state paths.

In order to compute our character formulas, we define the energy function satisfying that property:
\[
  H=
  \bordermatrix{
    \text{}&b_0& b_1 &b_2
    \cr b_0&1&1&1
    \cr b_1&0&0&1
    \cr b_2&-1&0&1}.
\]
We now use the theory of grounded and multi-grounded partitions to prove Proposition \ref{prop:characterlevel2}.

\subsection{The module $L(\Lambda_0+\Lambda_1)$}
The highest weight $\Lambda_0 + \Lambda_1$ has a constant ground state path. We can thus apply the theory of grounded partitions directly to the energy matrix $H$ to prove \eqref{eq:a11} of Proposition \ref{prop:characterlevel2}.

\begin{proof}[Proof of \eqref{eq:a11} of Proposition \ref{prop:characterlevel2}]
  We set $q= e^{-\delta}, c_0= e^{\alpha_1}, c_1=1, c_2 = e^{-\alpha_1}$, and let $\mathcal{P}^{\gg}_{c_1}$ denote the set of grounded partitions with colour set $\{c_0,c_1,c_2\}$, ground $c_0$ and relation $\gg$ given by
  $$k_{c_i}\gg l_{c_{j}} \Longleftrightarrow k-l\geq H(b_j\ot b_i).$$
  Then by Theorem \ref{th:char}, we have
  \begin{equation}
    \label{eq:char01}
    \sum_{\pi\in \mathcal{P}^{\gg}_{c_1}} C(\pi)q^{|\pi|}= \frac{e^{-(\Lambda_0 + \Lambda_1)}\ch L(\Lambda_0 + \Lambda_1)}{(q;q)_{\infty}}.
  \end{equation}

  Thus to compute the non-specialised character of $L(\Lambda_0 + \Lambda_1)$, we only need to find a nice expression for the generating functions of grounded partitions in $\mathcal{P}^{\gg}_{c_1}$. To do this, we note that the difference conditions given by the matrix $H$ correspond to the partial order
  \[\cdots \ll 0_{c_1} \ll \begin{array}{c} 1_{c_0} \\ 0_{c_{2}} \end{array} \ll 1_{c_1}\ll\begin{array}{c} 2_{c_{0}} \\ 1_{c_2}\end{array}\ll 2_{c_1}\ll\begin{array}{c} 3_{c_{0}} \\ 2_{c_2}\end{array}\ll \cdots .\]

  By definition, the grounded partitions in $\mathcal{P}^{\gg}_{c_1}$ must have $0_{c_1}$ as the last part, and the penultimate part should be different from $0_{c_1}$. The parts $k_{c_1}$ can repeat arbitrarily many times because $H(b_1 \ot b_1)=0$. Thus, including the last part $0_{c_1}$ generated by $c_1$, the parts coloured $c_1$ are generated by
  $$\frac{c_1}{(c_1 q;q)_{\infty}}.$$
  Now the parts coloured $c_0$ and $c_2$ can also appear several times, but always in sequences of the form
  \begin{equation*}\label{eq:twistedseq}
    \cdots \ll k_{c_0}\ll (k-1)_{c_2}\ll k_{c_0}\ll (k-1)_{c_2}\ll\cdots .
  \end{equation*}
  The generating function of such a sequence for a fixed integer $k \geq 1$ is given by
  $$ \frac{(1+c_{0}q^{k})(1+c_2q^{k-1})}{(1-c_0c_2q^{2k-1})},$$
  where the denominator generates pairs $(k_{c_0},(k-1)_{c_2})$ that can repeat arbitrarily many times, and the numerator accounts for the possibility of having an isolated $(k-1)_{c_2}$ on the left end of the sequence, or an isolated $k_{c_0}$ on the right end of the sequence. Multiplying this over all integers $k \geq 1$ gives the generating function for the parts coloured $c_0$ and $c_2$:
  $$\frac{(-c_0 q;q)_{\infty}(-c_2;q)_{\infty}}{(c_0c_2q;q^2)_{\infty}}.$$

  Thus the generating function for the grounded partitions in $\mathcal{P}^{\gg}_{c_1}$ is
  $$\sum_{\pi\in \mathcal{P}^{\gg}_{c_1}} C(\pi)q^{|\pi|}= \frac{c_1 (-c_0 q;q)_{\infty}(-c_2;q)_{\infty}}{(c_1q;q)_{\infty}(c_0c_2q;q^2)_{\infty}}.$$

  Substituting the correct values of $q, c_0, c_1, c_2$ into \eqref{eq:char01}, and using that $1/(q;q^2)_{\infty}=(-q;q)_{\infty}$, we obtain
  $$e^{-(\Lambda_0 + \Lambda_1)}\ch L(\Lambda_0 + \Lambda_1) = (-e^{-\alpha_0};q)_{\infty}(-e^{-\alpha_1};q)_{\infty}(-e^{-\delta};e^{-\delta})_{\infty},$$
  and \eqref{eq:a11} of Proposition \ref{prop:characterlevel2} is proved.
\end{proof}

\bigskip
It is interesting to note that in the purely combinatorial proof of Capparelli's identity given by Lovejoy and the first author in \cite{DousseCapa}, itself inspired by the combinatorial proof of Alladi, Andrews and Gordon \cite{AllAndGor}, they actually also compute the generating function for coloured partitions with difference conditions $\gg$, but with other initial conditions. Indeed, while in the present paper, the last part needs to be $0_{c_1}$, in their work, the smallest part could be any non-negative integer, which gave rise to the generating function
$$\frac{(-c_0;q)_{\infty}(-c_2;q)_{\infty}}{(c_1;q)_{\infty}(c_0c_2q;q^2)_{\infty}}.$$
Then, adding a staircase (i.e.\ a partition $1+2+3+ \cdots$) to such a partition -- which adds $1$ to all the differences between consecutive parts and therefore transforms $H$ into $H_0$ -- transforms it into a partition satisfying the difference conditions of Capparelli's identity. The proof concludes with some $q$-series computations to show the generating function that one obtains is an infinite product, and therefore the generating function for partitions with congruence conditions.

Again, our hope of finding a direct proof of Capparelli's identity via the theory of perfect crystals is not really fulfilled, because on the one hand the proof of Capparelli's identity uses the ``correct'' energy function but with the ``wrong'' initial conditions, and on the other hand, if we start from the ``correct'' energy function and initial conditions, then adding a staircase gives a generating function which is not an infinite product. However, it is still interesting and intriguing that a purely combinatorial proof, with no connection or background in crystal base theory, naturally leads to the study of the exact energy matrix that is meaningful in the theory of perfect crystals and grounded partitions. This hints yet again at a deeper connection between the approaches in combinatorics and representation theory, and between Capparelli's identity (which originally comes from vertex operator algebras) and crystal base theory.

\subsection{The module $L(2\Lambda_0)$}
We now turn to the module $L(2\Lambda_0)$, whose ground state path $\cdots \ot b_0 \ot b_2 \ot b_0 \ot b_2$ is of period $2$. Thus, we need the theory of multi-grounded partitions to prove \eqref{eq:a12} of Proposition \ref{prop:characterlevel2}.

\begin{proof}[Proof of \eqref{eq:a12}]
  As explained in Section \ref{sec:ground}, we need to choose a suitable divisor $D$ of $4$ such that
  $DH(\B_2\ot\B_2)\subset \Z$ and $\frac{1}{2} ( DH(b_0\ot b_2) + 2 DH(b_2\ot b_0)) \in \Z$.
  We have $H(b_2 \ot b_0) = -H (b_0 \ot b_2)=1$, so we can choose $D=2$, and define the relation $\gg$ as
  $$k_{c_i}\gg l_{c_{j}} \Longleftrightarrow k-l\geq DH(b_j\ot b_i).$$
  We want to study the multi-grounded partitions in $\mathcal{P}^{\gg}_{c_2 c_0}$. Using Definition \ref{deff:multiground} or Proposition \ref{prop:multigroundedparts}, we obtain that the two last parts of such partitions are fixed to be $(-1_{c_2},1_{c_0})$.

  Applying Theorem \ref{th:char} with $d=2$ and $D=2$ gives
  \begin{equation}
    \label{eq:tha210}
    \sum_{\pi\in \, _2^2\Pp_{c_2 c_0}^{\gg}} C(\pi)q^{|\pi|} = \frac{e^{-2\Ll_0}\ch L(2\Lambda_0)}{(q^2;q^2)_{\infty}},
  \end{equation}
  where $q=e^{-\delta/2}, c_0=e^{\alpha_1}, c_1=1, c_2 =e^{-\alpha_1}$.

  Recall that $_2^2\Pp_{c_2 c_0}^{\gg}$ is the set of multi-grounded partitions $\pi=(\pi_0,\dots,\pi_{2s-1},-1_{c_2},1_{c_0})$ with relation $\gg$ and ground $c_2, c_0$, \emph{having an even number of parts}, such that for all $k \in \{0, \dots , 2s-1\}$,
  \begin{equation}
    \label{eq:a21parity}
    \pi_k-\pi_{k+1} - 2H(b_{i_{k+1}}\ot b_{i_k}) \in 2 \Z_{\geq0},
  \end{equation}
  where $c(\pi_k)=c_{i_k}$ and $\pi_{2s}=-1_{c_2}$.

  Note that by \eqref{eq:a21parity} and the fact that $u^{(0)}=-1$, the multi-grounded partitions of $_2^2\Pp_{c_2 c_0}^{\gg}$ only have parts with odd sizes, as the differences between consecutive parts are even. Thus we can compute the generating function for partitions in $_2^2\Pp_{c_2 c_0}^{\gg}$ similarly to the last section, by noticing that, combined with \eqref{eq:a21parity}, $\gg$ is the following partial order on the set of coloured odd integers:

  \[\cdots  \ll \begin{array}{c} 1_{c_0} \\ -1_{c_{2}} \end{array} \ll 1_{c_1}\ll\begin{array}{c} 3_{c_{0}} \\ 1_{c_2}\end{array}\ll 3_{c_1}\ll\begin{array}{c} 5_{c_{0}} \\ 3_{c_2}\end{array}\ll \cdots .\]

  At first, let us temporarily forget the condition that the number of parts must be even, and let us compute the generating function.

  The parts $(2k+1)_{c_1}$ can repeat arbitrarily many times because $2H(b_1 \ot b_1)=0$, and thus are generated by
  $$\frac{1}{(c_1 q;q^2)_{\infty}}.$$

Moreover, as in the previous section, the parts coloured $c_0$ and $c_2$ always appear in sequences of the form
  \begin{equation*}
    \cdots \ll (2k+1)_{c_0}\ll (2k-1)_{c_2}\ll (2k+1)_{c_0}\ll (2k-1)_{c_2}\ll\cdots,
  \end{equation*}
and are generated (for $k \geq 1$) by
  $$\frac{(-c_0 q^3;q^2)_{\infty}(-c_2q;q^2)_{\infty}}{(c_0c_2q^4;q^4)_{\infty}}.$$
  Moreover, the fixed tail of the multi-grounded partitions under consideration is $(-1_{c_2},1_{c_0})$, so that for $k=0$, only an isolated $1_{c_0}$ can appear, but not $-1_{c_2}$, otherwise our partition would end with $(-1_{c_2},1_{c_0},-1_{c_2},1_{c_0})$, which is forbidden. Thus we should still multiply our generating function by $c_0c_2(1+c_0q)$.

  Combining all this, we obtain the generating function
  $$\frac{c_0c_2(-c_0 q;q^2)_{\infty}(-c_2q;q^2)_{\infty}}{(c_1 q;q^2)_{\infty}(c_0c_2q^4;q^4)_{\infty}}.$$

  Now, to take back into account the fact that there must be an even number of parts, we use \eqref{eq:evennumberparts}.
  Thus the multi-grounded partitions in $_2^2\Pp_{c_2 c_0}^{\gg}$ are generated by:
  \begin{align*}
    \sum_{\pi\in \ _2^2\Pp_{c_2 c_0}^{\gg}} C(\pi)q^{|\pi|} &= \mathcal{E}_{c_0,c_1,c_2} \left( \frac{c_0c_2(-c_0 q;q^2)_{\infty}(-c_2q;q^2)_{\infty}}{(c_1 q;q^2)_{\infty}(c_0c_2q^4;q^4)_{\infty}} \right)\\
                                                            &= \frac{c_0c_2}{(c_0c_2q^4;q^4)_{\infty}} \left( \frac{(-c_0 q;q^2)_{\infty}(-c_2q;q^2)_{\infty}}{(c_1 q;q^2)_{\infty}} + \frac{(c_0 q;q^2)_{\infty}(c_2q;q^2)_{\infty}}{(-c_1 q;q^2)_{\infty}}\right)\!.
  \end{align*}

  Substituting the correct values of $q, c_0, c_1, c_2$ in \eqref{eq:tha210} and simplifying completes the proof.
\end{proof}

\subsection{The module $L(2\Lambda_1)$}
We conclude with the character formula \eqref{eq:a13} for the module $L(2\Lambda_1)$, whose ground state path $\cdots \ot b_2 \ot b_0 \ot b_2 \ot b_0 $ is again of period $2$. Here we need to study the multi-grounded partitions in $_2^2\Pp_{c_0 c_2}^{\gg}$, and  apart from the fixed last parts which are now $(1_{c_0},-1_{c_2})$, everything works in the same way as in the previous section. We obtain the generating function
$$\mathcal{E}_{c_0,c_1,c_2} \left( \frac{c_0c_2(-c_0 q^3;q^2)_{\infty}(-c_2q^{-1};q^2)_{\infty}}{(c_1 q;q^2)_{\infty}(c_0c_2q^4;q^4)_{\infty}} \right),$$
and \eqref{eq:a13} follows thanks to Theorem \ref{th:char}.

\section{Conclusion}
It would be interesting to find obviously positive non-specialised character formulas and partition theoretic interpretations for standard modules of $A_n^{(1)}$ at all levels $\ell$. In the pair of papers \cite{DK19,DK19-2} of the first and third authors, a partition theoretic interpretation was given for $\ell=1$ and all $n$, and the non-specialised character formula of Kac--Peterson \cite{KacPeterson} was recovered purely combinatorially. In \cite{Meurman}, using vertex operator algebras, Meurman and Primc were able to interpret the Andrews--Gordon--Bressoud partition identities in terms of specialised characters of $A_1^{(1)}$ standard modules of all levels. In the present paper, using crystals, we give a different combinatorial interpretation of the same specialised characters at all levels, and of the non-specialised characters at level $2$.

The method of multi-grounded partitions, introduced by the first and third author in \cite{DKmulti} and summarised in Theorem \ref{th:char}, can, in theory, be applied to any Kac--Moody Lie algebra at any level, as long as the crystal and its energy function are understood. However, the technical difficulties in computing the generating functions of the corresponding multi-grounded partitions grow with the rank of the Lie algebra and the level of the modules, making it a non-trivial task from the combinatorial point of view as well. Nonetheless, we hope to make further progress towards the goal of understanding level $\ell$ standard modules of $A_n^{(1)}$ for all $\ell$ and $n$ in subsequent work.

\subsection*{Acknowledgements}
All three authors were partially supported by the project IMPULSION of IDEXLYON.
J.D. is funded by the ANR COMBIN\'e ANR-19-CE48-0011 and the SNSF Eccellenza grant number PCEFP2 202784.
I.K. is funded by the LABEX MILYON (ANR-10-LABX-0070) of Universit\'e de Lyon, within the program ``Investissements d'Avenir" (ANR-11-IDEX-0007) operated by the French National Research Agency (ANR).
The authors are grateful to the referee for their comments and suggestions which helped improve the quality of the paper.

\bibliographystyle{alpha}
\bibliography{biblio}

\end{document}